\documentclass[a4paper,11pt,english]{article}
\usepackage[utf8]{inputenc}
\usepackage[margin=1.25in]{geometry}
\usepackage{babel,amsmath,amssymb,amsthm,enumitem,setspace,authblk}
\usepackage[sort,comma]{natbib}
\usepackage[small,bf]{titlesec}
\titlelabel{\thetitle.\hspace{0.5em}}
\usepackage[colorlinks,allcolors=blue]{hyperref}

\newcommand{\as}{\text{a.s.}}
\renewcommand{\P}{\mathrm{P}}
\newcommand{\E}{\mathrm{E}}

\newcommand{\I}{\mathrm{I}}

\newcommand{\F}{\mathcal{F}}
\newcommand{\FF}{\mathbb{F}}
\newcommand{\R}{\mathbb{R}}
\renewcommand{\hat}{\widehat}
\renewcommand{\epsilon}{\varepsilon}
\newcommand{\bl}{\bar\lambda}
\newcommand{\bn}{\bar\nu}
\newcommand{\bW}{\bar W}
\newcommand{\hl}{\hat\lambda}

\newcommand{\simp}{\Delta}

\newtheorem{theorem}{Theorem}

\newtheorem{proposition}{Proposition}
\theoremstyle{definition}
\newtheorem{definition}{Definition}
\newtheorem{remark}{Remark}
\newtheorem{assumption}{Assumption}

\begin{document}
\title{On convergence of forecasts in prediction markets}
\author[1]{Nina Badulina}
\author[1]{Dmitry Shatilovich}
\author[2]{Mikhail Zhitlukhin%
\thanks{Corresponding author. E-mail: mikhailzh@mi-ras.ru.}}
\affil[1]{Lomonosov Moscow State University}
\affil[2]{Steklov Mathematical Institute, Moscow}
\date{}

\onehalfspacing
\maketitle

\begin{abstract}
We propose a dynamic model of a prediction market in which agents predict the values of a sequence of random vectors.
The main result shows that if there are agents who make correct (or asymptotically correct) next-period forecasts, then the aggregated market forecasts converge to the next-period conditional expectations of the random vectors.

\end{abstract}

\section{Introduction}

Prediction markets are artificial markets designed for extracting information scattered among traders.
In such markets, agents buy and sell contracts tied to outcomes of future events. 
The market price of a contract can be used as a forecast of the probability of the event.
Currently operating prediction markets include, for example, the Iowa Electronic Market and Metaculus (the latter positions itself not as a prediction market in the strict sense, but as a prediction aggregator).
In the past, prediction markets were used to forecast results of the US presidential elections before the development of population polls \citep{RhodeStrumpf04}.
It has been observed that prediction markets can provide accurate forecasts and outperform traditional statistical methods \citep{BergRietz03}.
Literature surveys on prediction markets can be found in \cite{TziralisTatsiopoulos07,HornIvens14}.

This paper provides a mathematical study of convergence of forecasts elicited from prediction markets to true probabilistic characteristics such as probabilities of random events or expected values of random variables. 
We consider a dynamic game in which agents predict the values of a sequence of random vectors.
We show that if there are agents who make correct (or asymptotically correct) next-period forecasts, then the market mechanism rewards those agents, which results in that the aggregated market forecasts converge to the next-period conditional expectations of the random vectors as time $t\to\infty$. 
Agents who make incorrect predictions eventually have vanishing impact on the aggregated forecasts. 

The problem of accuracy of forecasts in predication markets has been extensively studied in the literature.
\cite{WolfersZitzewitz04,Manski06,WolfersZitzewitz06a}, among others, considered static (single-period) models. Dynamic models were studied by \cite{Pennock04,BeygelzimerLangford12,KetsPennock14,BottazziGiachini19}. 
The majority of theoretical works in this direction obtain results under the assumption that agents choose strategies from a particular class, e.g.\ maximize certain utility functions.
From a practical point of view, this might be not desirable since this requires to know unobservable agents' characteristics in order to interpret market forecasts.
The novelty of our model consists in that we show that convergence to true conditional expectations takes place regardless of strategies used by agents.
The only assumption we impose is that the market contains agents who make correct forecasts.
(Let us clarify: we do not prove that a prediction market can provide a more accurate forecast than any of its agent. Our goal is to show that it is possible to extract correct forecasts without knowing who of the agents is right.) 

The paper is organized as follows.
In Section~\ref{sec:model} we describe the model. Section~\ref{sec:main} presents the main results of the paper, first in the general case and then  provides a refinement for a stationary version of the model.
These results are obtained for a discrete-time setting with one-shot forecasts of the next-period random vectors.
In Section~\ref{sec:extended}, we show how the basic model can be extended to continuous-time to allow dynamic forecasts.

\section{The model}
\label{sec:model}

We consider a game played by $M$ agents who pursue the goal to earn money by making forecasts of the value of an $N$-dimensional random vector in each round of an infinite sequence of rounds. 

The game is played in discrete time.
The information flow is modelled by a probability space with a discrete-time filtration $\FF=(\F_t)_{t\ge 0}$.
The random vectors which the agents try to predict form an adapted sequence $X=(X_t)_{t\ge 1}$, where $X_t=(X_t^1,\dots,X_t^N)$.
The components $X_t^n$ are non-negative and sum up to 1 for each $t$, so $X_t$ assumes values in the standard $N$-simplex $\simp^N = \{x\in\R_+^N : x^1+\ldots+x^N=1\}$.
A basic example is a vector of indicators of $N$ random events forming a partition of the probability space; for other examples, see Remark~\ref{variants} below.

We will always assume that the next-period conditional expectations of $X_t^n$ are bounded away from zero, i.e.\ there exists a constant $\epsilon>0$ such that for all $t\ge 0$ and $n=1,\dots,N$
\begin{equation}
\label{positive-x}
\E(X_{t+1}^n \mid \F_t) \ge \epsilon.
\end{equation}

The agents enter the game at time $t=0$ with strictly positive non-random initial wealth $W_0^m$, $m=1,\dots,M$.
The wealth $W_t^m$ at further moments of time is random and determined by the dynamics described below.

At each moment of time $t\ge 0$, an agent can make bets on the components of the vector $X_{t+1}$.
Let $h_t^{mn}$ denote the amount of money bet by the $m$-th agent on the $n$-th component. 
These amounts are assumed to be non-negative and their sum must not exceed the wealth of the agent ($\sum_{n=1}^M h_t^{mn} \le W_t^m$).
By $h_t^{m0} := W_t^m - \sum_{n=1}^n h_t^{mn}$, we denote the portion wealth of this agent which remains after he/she places the bets.

At time $t+1$, the value of $X_{t+1}$ becomes known and the agents get paid according to the following scheme.
The pool of wagered money $w_t := \sum_{mn}h_t^{mn}$ is divided into $N$ parts $V_{t+1}^n := w_t X_{t+1}^n$ and each part is distributed among the agents proportionally to their bets on the corresponding component of $X_{t+1}$.

As a result, the dynamics of an agent's wealth is governed by the equation
\begin{equation}
\label{capital-h}
W_{t+1}^m = \sum_{n=1}^N \frac{h_t^{mn}}{\sum_{k=1}^M  h_t^{kn}} V_{t+1}^n + h_t^{m0}.
\end{equation}
In what follows, instead of the variables $h_t^{mn}$, it will be more convenient to work with the variables $\nu_t^m = \sum_{n=1}^N h_t^{mn} / W_t^m$ and $\lambda_t^{mn} = h_t^{mn} / \sum_{n=1}^N h_t^{mn}$.
The former represents the proportion of money that agent $m$ allocates for betting, and the latter are equal to the fraction of money bet on the $n$-th component of $X_{t+1}$.
Then equation \eqref{capital-h} takes the form
\begin{equation}
\label{capital}
W_{t+1}^m = \sum_{k=1}^M \nu_t^k W_t^k 
\cdot \sum_{n=1}^N \frac{\nu_t^m\lambda_t^{mn}  W_t^m}{\sum_{k=1}^M \nu_t^k\lambda_t^{kn}  W_t^k} X_{t+1}^n 
+ (1-\nu_t^m) W_t^m.
\end{equation}

The pair of sequences $\sigma=(\nu,\lambda)$ defines the strategy of an agent. 
The sequence $\nu=(\nu_t)_{t\ge0}$ is scalar-valued with values in $[0,1]$, while $\lambda=(\lambda_t)_{t\ge 0}$ is vector-valued with values in $\Delta^N$.
We assume that these sequences may depend on the random outcome $\omega$ in an $\F_t$-measurable way, i.e.\ $\nu_t=\nu_t(\omega)$ and $\lambda_t=\lambda_t(\omega)$ are $\F_t$-measurable functions. 

Equations \eqref{capital-h} and \eqref{capital} make sense only if the denominators in these formulas do not vanish. 
The following proposition provides a sufficient condition for that. 

\begin{proposition}
Suppose that some agent (say, agent $m$) uses a strategy $\sigma^m=(\nu^m,\lambda^m)$ such that  $\nu_t^m>0$ and $\lambda_t^{mn}>0$ for all $t\ge0$ and $n=1,\dots,N$.
Then $W_t^m > 0$ and $\sum_{k=1}^M \nu_t^k \lambda_t^{kn}W_t^k>0$ for all $t\ge 0$.
\end{proposition}

The proof obviously follows from formula \eqref{capital} and the assumption that the initial wealth is strictly positive. 
Hereinafter, we will always assume that all strategy profiles under consideration satisfy the condition of this proposition.

Observe also that if the denominators in \eqref{capital} are positive, then the total wealth $\bar W_t = \sum_{m=1}^M W_t^m$ remains constant.
Therefore, without loss of generality, we can assume that $\bW_t \equiv 1$ and hence $W_t \in \simp^M$. 

Define the \textit{wealth-weighted strategy} of the agents as $\bar\sigma = (\bn, \bl)$, where
\[
\bn_t = \sum_{m=1}^M \nu_t^m W_t^m, \qquad
\bl_t^n = \frac{1}{\bn_t}\sum_{m=1}^M \nu_t^m \lambda_t^{mn} W_t^m.
\]
Note that $\bn_t\in[0,1]$ and $\bl_t \in \Delta^N$. 
The vector $\bl_t$ can be represented as 
\[
\bl_t^n = \frac{\sum\limits_{m=1}^M h_t^{mn}}{\sum\limits_{m=1}^M \sum\limits_{i=1}^N h_t^{mi}},
\]
and therefore has a simple interpretation: $\bl_t^n$ is the proportion of money bet by all the agents on the $n$-th component of $X_{t+1}$ in the total pool of wagered money.

\begin{definition}
We call $\bl_t$ the \emph{market forecast} of the vector $X_{t+1}$.
\end{definition}
The goal of the remaining part of the paper will be to investigate when $\bl_t$ converges to $\E(X_{t+1}\mid\F_t)$ as $t\to\infty$.
The following notion will play the key role in establishing the convergence.

\begin{definition}
We call a strategy $\sigma^m$ \emph{survival}, if for any strategy profile $\Sigma=(\sigma^1,\dots,\sigma^M)$, which contains this strategy, and any vector of initial wealth $W_0$, it holds that%
\[
\liminf_{t\to\infty} W_t^m > 0\ \as
\]
\end{definition}

From the symmetry of the model, it is clear that the property of survival does not depend on the number $m$ of an agent it is applied to.

The main results in the next section show that the presence of agents using survival strategies  ensures that the marked forecasts converge to the conditional expectations.

\begin{remark}
The notion of a survival strategy is borrowed from the literature in evolutionary finance (see, e.g., \cite{BlumeEasley92,EvstigneevHens+02,AmirEvstigneev+13}) and the model we consider is a modification of the model of a financial market with short-lived assets of \cite{AmirEvstigneev+13}.
For generalizations of this model, see \cite{Zhitlukhin22,Zhitlukhin23,EvstigneevTokaeva23}.
\end{remark}

\begin{remark}
\label{variants}
Let us give examples how one can choose vectors $X_t$ to obtain estimates of various characteristics.
Recall that the components of $X_t$ must sum up to~1.

As mentioned earlier, if random events $A_t^n \in \F_t$ form a partition of the underlying probability space for each $t$, we can put $X_t^n = \I(A_t^n)$, where $\I(\cdot)$ denotes the indicator function.
Then $\bl_t^n$ provides an estimate of the conditional probability $\P(A_{t+1}^n\mid \F_t)$.

If random events $A_t^n$, $n=1,\dots,N$, are not disjoint, the same probability can be estimated by $N \bl_t^n$, if we put $X_t^n = \frac 1N \I(A_t^n)$ and $X_t^{N+1} = 1 - \sum_{n=1}^N X_t^n$ (the additional component is needed to ensure that $X_t$ takes values in $\simp^{N+1}$).

To estimate conditional expectations of a sequence of bounded random variables $\xi_t$, consider the two-dimensional vectors $X_t$ with $X_t^1 = (\xi_t-a_t)/b_t$ and $X_t^2 = 1 -\xi_t$, where $a_t = \mathrm{ess\,inf}_\omega \xi_t(\omega)$, $b_t = \mathrm{ess\,sup}_\omega \xi_t(\omega) - a_t$.
Then $\E(\xi_{t+1}\mid \F_t)$ is estimated by $a_t + b_t \bl_t^1$.
If the random variables are not bounded, one can first truncate them at some levels and then use this construction.

Conditional moments can be estimated in a similar way.
Assume that $\xi_t$ are bounded, and, without loss of generality, take on values in $[0,1]$.
Put $X_t = (X_t^1,\dots, X_t^{N+1})$ with $X_t^n = (\xi_t)^n/N$ for $n=1,\dots,N$ and $X_t^{N+1} = 1 - \sum_{n=1}^N X_t^n$.
Then $N \bl_t^n$ provides an estimate of the $n$-th conditional moment $\E(\xi_{t+1}^n\mid \F_t)$ for $n=1,\dots,N$.
For example, for $N=2$, we can estimate the conditional expectation and conditional variance by $2\bl_t^1$ and $2\bl_t^2 - 4(\bl_t^1)^2$, respectively.
\end{remark}

\section{Main results}
\label{sec:main}

\subsection{The general case}
Denote by $\mu_t=(\mu_t^1,\dots,\mu_t^N)$, $t\ge0$, the vector of conditional expectations
\[
\mu_t^n = \E(X_{t+1}^n\mid \F_t).
\]
Note that $\mu_t^n\ge \epsilon>0$ by assumption~\eqref{positive-x}.

In what follows, let us denote by $\|x\|$ the Euclidean norm of a vector $x\in\R^N$. If $\xi$ is a random vector, then $\|\xi\|$ will denote the random norm, i.e.\ $\|\xi(\omega)\|$.

\begin{theorem}
\label{th1}
1) Consider a strategy $\sigma=(\nu,\lambda)$ such that $\lambda_t^n>0$ for all $t\ge0$ and $n=1,\dots,N$. Then $\sigma$ is survival if
\begin{equation}
\label{survival-condition}
\sum_{t=0}^\infty \|\lambda_t - \mu_t\|^2 < \infty\ \as
\end{equation}
Moreover, for a strategy $\sigma$ such that $\lambda_t^n>0$ and $\nu_t$ is bounded away from zero, condition \eqref{survival-condition} is also necessary for survival.

\smallskip
2) Suppose that in a strategy profile, some agent uses a survival strategy $\sigma=(\nu,\lambda)$ such that its components $\lambda_t^n$ and $\nu_t$ are bounded away from zero.
Then 
\[
\sum_{t=0}^\infty \|\bl_t-\mu_t\|^2 < \infty\ \as
\]
\end{theorem}

\begin{proof}
1) Let agent $m=1$ use a strategy $\sigma=(\nu,\lambda)$ satisfying \eqref{survival-condition} and having strictly positive components $\lambda_t^n$. 
Define the sequences $U=(U_t)_{t\ge0}$ and $Z=(Z_t)_{t\ge0}$:
\[
U_0 = 0,\qquad
U_{t+1} = U_t + \nu_t \sum_{n=1}^N \mu_t^n\ln\frac{\mu_t^n}{\lambda_t^n}, \qquad
Z_t = \ln W_t^1 + U_t.
\]
Note that $U$ is non-decreasing by Gibbs' inequality. From reverse Pinsker's inequality (see, e.g., \cite{SasonVerdu15}), we find that
\[
\sum_{n=1}^N \mu_t^n\ln\frac{\mu_t^n}{\lambda_t^n} \le \frac{|\mu_t-\lambda_t|^2}{2\min_n \lambda_t^n},
\]
where $|\cdot|$ denotes the $\ell^1$-norm of a vector.
Then, in view of condition \eqref{survival-condition} and the equivalence of norms on $\R^N$, the sequence $U_t$ converges \as\ to a finite limit as $t\to\infty$.

Let us show that $Z$ is a local submartingale\footnote{Technical results on martingales needed below can be found in, e.g., \cite{Shiryaev19b}.}.
To this end, it is enough to show that $\E(Z_{t+1}^+\mid \F_t)<\infty$ and $\E(Z_{t+1} - Z_t \mid \F_t) \ge 0$.
The first inequality  follows from that $Z_{t+1} \le U_{t+1}$ and $U_{t+1}$ is $\F_t$-measurable.

Observe that the wealth dynamics equation \eqref{capital} for agent $1$ can be written as 
\[
\frac{W_{t+1}^1}{W_t^1} = \nu_t \sum_{n=1}^N \frac{\lambda_t^n}{\bl^n_t} X_{t+1}^n
+ 1-\nu_t.
\]
From here, we find
\begin{multline}
\label{logW-bound}
\ln W_{t+1}^1 - \ln W_t^1
= \ln\left(\nu_t \sum_{n=1}^N X_{t+1}^n\frac{\lambda_t^n}{\bl_t^n} + 1-\nu_t \right) \ge \nu_t \ln\left(\sum_{n=1}^N X_{t+1}^n\frac{\lambda_t^n}{\bl_t^n}\right) \\
\ge \nu_t \sum_{n=1}^N X_{t+1}^n\ln \frac{\lambda_t^n}{\bl_t^n},
\end{multline}
where in both inequalities we used the concavity of the logarithm and considered its arguments as convex combinations with coefficients $\nu_t$ and $1-\nu_t$ in the first inequality, and with coefficients $X_{t+1}^n$ in the second one.
Consequently,
\[
\E (\ln W_{t+1}^1 - \ln W_t^1 \mid \F_t) \ge \nu_t \sum_{n=1}^N \mu_t^n\ln\frac{\lambda_t^n}{\bl_t^n}.
\]
Adding $U_{t+1}-U_t$ to this inequality, we find that
\[
\E(Z_{t+1} - Z_t \mid \F_t) \ge \nu_t \sum_{n=1}^N \mu_t^n\ln\frac{\mu_t^n}{\bl_t^n} \ge 0
\]
by Gibb's inequality.
Thus, $Z$ is a local submartingale.

Since $Z$ is bounded from above by a predictable sequence ($Z_t\le U_t$), it is actually a true submartingale and there exists a finite limit $\lim_{t\to\infty} Z_t$. 
As $U$ converges, the limit of $\ln W_t^1$ also exists, which implies that $\lim_{t\to\infty} W_t^1 > 0$.
Hence, the strategy $\sigma$ is survival.
This proves the first part of claim 1.

Now suppose a strategy $\sigma$ is survival.
Then it must survive in the strategy profile  $\Sigma=(\hat\sigma, \sigma)$, where $\hat\sigma=(\hat\nu,\hl)$, $\hat\nu_t\equiv 1$, $\hl_t=\mu_t$.
For this strategy profile, put $Z_t = \ln W_t^1$.
As shown above, $Z$ is a submartingale.
Let $A$ denote its compensator, i.e. $A_t = \sum_{s=0}^{t-1} \E(Z_{s+1}-Z_s \mid \F_s)$, $A_0=0$. 

From \eqref{logW-bound}, we find
\[
\Delta A_t 
\ge \sum_{n=1}^N \mu_t^n\ln \frac{\mu_t^n}{\bl_t^n} 
\ge \frac{|\mu_t - \bl_t|^2}{2} = \left(\frac{\nu_t W^2_t}{W^1_t + \nu_tW_t^2}\right)^2 \frac{|\mu_t-\lambda_t|^2}{2},
\]
where the second inequality follows from Pinsker's inequality.

Since $Z_t$ converges, $A_t$ also converges.
Since $\nu_t\ge \epsilon'>0$ and $\liminf_{t\to\infty} W_t^2>0$, the series $\sum_{t=0}^\infty |\mu_t-\lambda_t|^2$ converges, which proves that condition \eqref{survival-condition} is necessary for survival.

\medskip
2) Consider a strategy profile in which agent 1 uses a survival strategy $\sigma=(\nu,\lambda)$ with components $\lambda_t^n$ and $\nu_t$ bounded away from zero.
As in the proof of the first claim, we find that $Z_t = \ln W_t^1 + U_t$ is a submartingale and $U_t$ converges. 
Using \eqref{logW-bound}, we can estimate the compensator $A_t$ of $Z_t$ by
\[
\Delta A_t \ge \nu_t\sum_{n=1}^N \mu_t^n \frac{\mu_t^n}{\bl_t^n} \ge \frac{\nu_t}{2} |\mu_t-\bl_t|^2.
\]
From the convergence of the compensator and the assumption that $\nu_t$ is bounded away from zero, we obtain the convergence of the series $\sum_{t=0}^\infty |\mu_t - \bl_t|^2$, which implies the second claim of the theorem.
\end{proof}

\subsection{The stationary case}

Now we consider a particular case of the general model in which the game is driven by a stationary sequence and prove a stronger version of Theorem~\ref{th1}.

Let us introduce the following assumption.

\begin{assumption}
\label{stationary-assumption}
Let the sequence of random vectors $X_t$ satisfy the following properties.
	
\begin{enumerate}[label=(\alph*),widest=a,itemsep=0mm,topsep=0mm,leftmargin=*]
\item There exists a stationary ergodic Markov sequence $s=(s_t)_{t\ge1}$ with values in some measurable space $\mathcal{S}$ such that all $X_t$ functionally depend on $s_t$, i.e.\ $X_t = X(s_t)$, where $X\colon\mathcal{S} \to \simp^N$ is a non-random function (the sequence $s_t$ can be interpreted as a sequence of ``states of the world'');
		
\item the components $X_t^n$, $n=1,\dots,N$, are not conditionally linearly dependent, i.e.\ any non-trivial linear combination of them with $\sigma(s_{t-1})$-measurable coefficients is not a null random variable.
\end{enumerate}
\end{assumption}

Let $\mu(s) = \E(X_{t+1} \mid s_t = s)$. 
According to the previous section, any strategy $\sigma=(\nu,\lambda)$ with $\lambda_t = \mu(s_t)$ is survival.
Our next goal will be to prove the following stronger result.
Suppose a strategy profile consists of strategies $\sigma=(\nu,\lambda)$ with $\nu$ and $\lambda$ functionally dependent on $s_t$, i.e.\ $\nu_t=\nu(s_t)$, $\lambda_t=\lambda(s_t)$, and $\nu_t$ bounded away from zero.
We will show that if at least one agent uses a strategy with $\lambda_t = \mu(s_t)$, then the relative wealth of any agent with a strategy such that $\P(\lambda_t \neq \mu(s_t)) > 0$ vanishes asymptotically.

\begin{theorem}
Let Assumption~\ref{stationary-assumption} hold.
Consider a strategy profile $\Sigma=(\sigma^1,\dots,\sigma^N)$ in which every agent uses a strategy $\sigma^m=(\nu^m,\lambda^m)$ such that the components $\lambda_t^{mn}$ and $\nu_t^m$ are bounded away from zero and can be represented in the form $\nu_t^m = \nu^m(s_t)$, $\lambda_t^m = \lambda^m(s_t)$ for non-random functions $\nu^m$, $\lambda^m$.
	
Suppose $\lambda^1(s)=\mu(s)$.
Then $\lim_{t\to\infty} W_t^m = 0$ \as\ for any agent $m$ such that $\P(\lambda^m(s_t) \neq \mu(s_t)) > 0$.
\end{theorem}

\begin{proof}
Without loss of generality, we will assume that the underlying filtration $(\F_t)_{t\ge 0}$ is generated by the sequence $s_t$.

We have $\liminf_{t\to\infty} W_t^1 >0$ since the strategy of agent 1 is survival. Hence it will be enough to show that $\lim_{t\to\infty} W_t^1/W_t^m = \infty$.
For that end, we will prove the inequality
\[
\liminf_{t\to\infty} \frac{1}{t} \ln \frac{W_t^1}{W_t^m} > 0\ \as
\]
Denote $D_t = \ln(W_t^1/W_t^m) - \ln (W_{t-1}^1/W_{t-1}^m)$.
Then we have the obvious representation
\begin{equation}
\label{logW}
\frac{1}{t} \ln \frac{W_{t+1}^1}{W_{t+1}^m}
= \frac1t \ln\frac{W_0^1}{W_0^m}
+ \frac 1t \sum_{u=0}^t(D_{u+1} - \E(D_{u+1}\mid\F_u))
+ \frac1t\sum_{u=0}^t \E(D_{u+1}\mid \F_u).
\end{equation}
We will show that the limit of the second term in the right-hand side is zero, and the limit inferior of the third term is strictly positive.

As for the second term, we have
\begin{equation}
\label{D-representation}
D_{t+1} = \ln\left(\nu_t^1 \sum_{n=1}^N X_{t+1}^n\frac{\mu^n_t}{\bl_t^n} + 1-\nu_t^1\right) - \ln\left(\nu_t^m \sum_{n=1}^N X_{t+1}^n\frac{\lambda^{mn}_t}{\bl_t^n} + 1-\nu_t^m\right).
\end{equation}
Since the components of all strategies are bounded away from zero, $D_t$ is uniformly bounded, so the sequence $\xi_t  = \sum_{u=0}^t(D_{u+1} - \E(D_{u+1}\mid\F_u))$ is a zero-mean martingale with bounded increments.
Then $\lim_{t\to\infty} \xi_t/t = 0$ \as\ by the strong law of large numbers for martingales.

Now consider the third term in the right-hand side of \eqref{logW}.
In order to prove that its limit inferior is strictly positive, we will show that there exists a random sequence $V_t$ and a function $g(s)$ such that $\lim_{t\to\infty} V_t = 0$ \as, $\E g(s_t) > 0$, and for all $t\ge 1$
\begin{equation}
\label{D-bound}
\E(D_{t+1}\mid \F_t) \ge V_t + g(s_t).
\end{equation}
If this is so, then the rest of the proof will follow from the convergence of the Ces\`aro sum (applied to $V_t$) and the ergodic theorem (applied to $g(s_t)$).

Using the concavity of the logarithm and Pinsker's inequality in \eqref{D-representation}, one can see that the conditional expectation of the first logarithm is non-negative, and therefore
\begin{multline*}
\E(D_{t+1}\mid \F_t)
\ge - \E\left(\ln\left(\nu_t^m \sum_{n=1}^N X_{t+1}^n\frac{\lambda^{mn}_t}{\bl_t^n}
+ 1-\nu_t^m \right)\;\Bigg|\; \F_{t}\right) \\\ge
-\E \Biggl( \ln \Biggl(\nu_t^m \sum_{n=1}^N X^n_{t+1} \frac{\lambda^{mn}(s_t)}{\mu^n(s_t)} +1-\nu_t^m\Biggr) \;\Bigg|\; \F_t\Biggr) - 
\ln \left(\max_{n=1,\dots,N} \frac{\mu^n(s_t)}{\bl_t^n}\right).
\end{multline*}
This implies that inequality \eqref{D-bound} holds for the sequence
\[
V_t =  - \ln \left(\max_{n=1,\dots,N} \frac{\mu^n(s_t)}{\bl_t^n}\right)
\]
and the function
\[
g(s) = -\E \Biggl( \ln \Biggl(\nu^m(s) \sum_{n=1}^N X^n(s_{t+1}) \frac{\lambda^{mn}(s)}{\mu^n(s)} +1-\nu^m(s)\Biggr) \;\Bigg|\; s_t=s\Biggr).
\]
By Theorem~\ref{th1}, we have $\lim_{t\to\infty}\|\mu(s_t) - \bl_t\|=0$ \as\ and hence $\lim_{t\to\infty} V_t = 0$.
Let us show that $\E g(s_t)>0$.

Observe that Assumption \ref{stationary-assumption}(b) implies that for any $t\ge0$ and any $\sigma(s_t)$-measurable random variables $c_1,\dots,c_N$ such that $\P(c_i\neq c_j)>0$ for at least one pair $(i,j)$, the random variable $c_0 = \sum_{n=1}^N c_n X_{t+1}^n$ is not $\sigma(s_t)$-measurable.
Indeed, otherwise we would have $\sum_{n=1}^N (c_n-c_0) X_{t+1}^n = 0$ and, hence, all $c_n$ would be equal to $c_0$ by Assumption \ref{stationary-assumption}(b), which is a contradiction.

Consequently, using Jensen's inequality, we find
\begin{multline*}
\E g(s_t)
= -\E  \ln \Biggl(\nu^m(s_t) \sum_{n=1}^N X^n(s_{t+1}) \frac{\lambda^{mn}(s_t)}{\mu^n(s_t)} +1-\nu^m(s_t)\Biggr) \\
> -\ln \E \Biggl(\nu^m(s_t) \sum_{n=1}^N X^n(s_{t+1}) \frac{\lambda^{mn}(s_t)}{\mu^n(s_t)} +1-\nu^m(s_t)\Biggr) = 0,
\end{multline*}
where the strict inequality takes place in view of the strict concavity of the logarithm and that its argument is not constant.
Thus, $\E g(s_t) > 0$, which finishes the proof.
\end{proof}

\section{Extension: a model with continuous-time forecasts}
\label{sec:extended}

The model considered above has a drawback: it provides only a one-shot market forecast of the vector $X_{t+1}$ based on the information derived from the agents' bets at time $t$, but does not take into account additional information that may appear in the interval $(t, t+1)$.
Now we propose a generalization of our model that addresses this issue.

Let the underlying probability space be equipped with a continuous-time filtration $\FF=(\F_t)_{t\ge0}$.
As before, we will assume that the values of the vectors $X_t$ become known at integer moments of time.
However, in the new model, the agents may place bets continuously in time according to the scheme described below.

By \emph{round $t$} of the game, let us call the interval of time $[t,t+1)$.
In the beginning of a round, the agents must choose proportions $\nu_t^m$ of their wealth they stake in this round.
The amount of money $\nu_t^m W_t^m$ is locked on agent $m$'s account.
Then the operator of the game gradually debits the agent's account and places bets on the components of the vector $X_{t+1}$ according to the agent's strategy $\lambda_s^m = (\lambda_s^{m1},\dots,\lambda_s^{mN})$, $s\in[t,t+1)$.
The process $\lambda_s^{mn}$ specifies the intensity of betting on the $n$-th component of $X_{t+1}$.
The agent is allowed to change the vector $\lambda_s^m$ dynamically during the round.
The processes $\lambda^{mn}_s$ must be adapted to the filtration $\FF$.

Based on this description, define the wealth dynamics in the game by the equation
\begin{equation}
\label{capital-extended}
W_{t+1}^m = \sum_{k=1}^M \nu_t^k W_t^k \cdot \sum_{n=1}^N X_{t+1}^n \int_t^{t+1}  \frac{\lambda_s^{mn}\nu_t^m W_t^m}{\sum_{k=1}^M \lambda_s^{kn}\nu_t^k W_t^k} \,d G(s) + (1-\nu_t^m)W_t^m, \quad t=0,1,\dots,
\end{equation}
where $G(t)$ is a function that specifies the speed with which the agents' accounts are debited by the operator.
We will assume that $G(t)$ is non-random, non-decreasing, continuous on the right with left-hand limits, and $G(t) = t$ for integer $t$.
The integral in formula \eqref{capital-extended} is understood in the Lebesgue-Stieltjes sense over the interval $[t,t+1)$. 

The function $G(t)$ is chosen by the operator and is made known to all agents.
Two basic examples are $G(t) = t$ and $G(t) = [tk]/k$, where $k\in\mathbb{N}$ is a fixed constant.
In the first case, the bets are placed continuously, while in the second one they are placed at moments of time $t+i/k$, $i=0,\dots,k$.
For $k=1$, we obtain the model from Section~\ref{sec:model}.

Let us define the \textit{wealth-weighted strategy} of the agents as $\bar\sigma = (\bn, \bl)$ with
\[
\bn_t = \sum_{m=1}^M \nu_t^m W_t^m,\quad t=0,1,\dots, \qquad
\bl_t^n = \frac{1}{\bn_{[t]}}\sum_{m=1}^M \nu_{[t]}^m \lambda_t^{mn} W_{[t]}^m, \quad t\in \R_+.
\]
We call $\bl_s$, $s\in[t,t+1)$, the \emph{market forecast process} of the vector $X_{t+1}$.
We call a strategy $\sigma=(\nu,\lambda)$ \emph{survival}, if in any strategy profile, in which agent $m$ uses this strategy, it holds that $\liminf_{t\to\infty} W_t^m > 0$.

Define the continuous-time conditional expectation process
\[
\mu_t = \E(X_{[t]+1} \mid \F_t).
\]
In what follows, let us assume that $\mu_t^n\ge\epsilon>0$ for all $t\ge 0$, $n=1,\dots,N$.

The next result is an analogue of Theorem~\ref{th1}.

\begin{theorem}
1) Consider a strategy $\sigma=(\nu,\lambda)$ with components $\lambda_t^n$ bounded away from zero.
Then $\sigma$ is survival if
\begin{equation}
\label{survival-condition-ext}
\int_0^\infty \|\lambda_t - \mu_t\|^2 d G(t)< \infty\ \as
\end{equation}
Moreover, if both $\lambda_t^n$ and $\nu_t$ are bounded away from zero, then condition \eqref{survival-condition-ext} is also necessary for $\sigma$ to be survival.

\smallskip
2) Suppose some agent uses a survival strategy $\sigma=(\nu,\lambda)$ with components $\nu_t$ and $\lambda_t^n$ bounded away from zero.
Then 
\[
\int_0^\infty \|\bl_t-\mu_t\|^2 d G(t)< \infty\ \as
\]
\end{theorem}

\begin{proof}
By and large, we repeat the proof of Theorem~\ref{th1}.

1) Suppose agent $m=1$ uses a strategy $\sigma=(\nu,\lambda)$ with $\lambda_t^n$ bounded away from zero and such that \eqref{survival-condition-ext} holds. 
It will be enough to show that the sequence $Z_t = \ln W_t^1 + U_t$ is a local submartingale, where $U_t$ is a non-decreasing convergent sequence defined by
\[
U_{t+1} = U_t + \nu_t \sum_{n=1}^N \int_t^{t+1} \mu_s^n\ln\frac{\mu_s^n}{\lambda_s^n}\,d G(s), \qquad
U_0 = 0.
\]
Using that the logarithm is concave, we obtain the bound
\begin{multline*}
\ln W_{t+1}^1 - \ln W_t^1 = \ln\left(\nu_t \sum_{n=1}^N X_{t+1}^n\int_t^{t+1}\frac{\lambda_s^n}{\bl_s^n}\,dG(s) + 1-\nu_t \right) \\
\ge \nu_t \ln\left(\sum_{n=1}^N X_{t+1}^n\int_t^{t+1}\frac{\lambda_s^n}{\bl_s^n}\,dG(s)\right) 
\ge \nu_t \sum_{n=1}^N X_{t+1}^n\int_t^{t+1}\ln\frac{\lambda_s^n}{\bl_s^n}\,dG(s).
\end{multline*}
A standard argument using the tower property of conditional expectation then yields
\[
\E(Z_{t+1} - Z_t \mid \F_t)
\ge \nu_t\,\E \left(\int_t^{t+1}\sum_{n=1}^N \lambda_s^n\ln\frac{\lambda_s^n}{\bl_s^n}\,dG(s) \;\Bigg|\; \F_{t} \right)
\ge 0.
\]
Therefore, $Z$ is a local submartingale, which implies that $\sigma$ is a survival strategy by the same reasoning as in the proof of Theorem~\ref{th1}.

If $\sigma$ is a survival strategy, let us place it in the profile $\Sigma =(\hat\sigma, \sigma)$, where $\hat\sigma=(\hat\nu,\hl)$, $\hat\nu_t\equiv 1$, $\hl_t=\mu_t$.
Then the sequence $Z_t := \ln W_t^1$ is a convergent submartingale.
Using Pinsker's inequality, we obtain the following bound for its compensator:
\begin{multline*}
\Delta A_t 
\ge \E \left(\int_t^{t+1}\frac{|\mu_s - \bl_s|^2}{2}\,dG(s)\;\Bigg|\; \F_{t} \right) \\
= \left(\frac{\nu_t W^2_t}{W^1_t + \nu_tW_t^2}\right)^2\E \left(\int_t^{t+1} \frac{|\mu_s-\lambda_s|^2}{2}\,dG(s)\;\Bigg|\; \F_{t} \right).
\end{multline*}
The convergence of $Z$ and $A$ then implies the convergence of the series
\begin{equation}
\label{another-compensator}
C_t = \sum_{s=0}^{t-1}\E\left(\int_s^{s+1} |\mu_u-\lambda_u|^2\,dG(u)\;\Bigg|\; \F_{s}\right).
\end{equation}
Observe that $C_t$ is the compensator of the discrete-time non-negative submartingale $B_t = \int_0^t |\mu_u-\lambda_u|^2\,dG(u)$, and hence this submartingale converges.
This yields the second part of the first claim of the theorem.

\medskip
2) If agent 1 uses a survival strategy $\sigma=(\nu,\lambda)$ with $\lambda_t^n$ and $\nu_t$ bounded away from zero, then $Z_t = \ln W_t^1 + U_t$ is a convergent submartingale and we have the following bound  for its compensator obtained in the same way as before:
\[
\Delta A_t  \ge \frac{\nu_t}{2} \E\left(\int_t^{t+1}|\mu_s-\bl_s|^2\,dG(s)\;\Bigg|\; \F_{t}\right).
\]
Arguing as in \eqref{another-compensator}, we find that $\int_0^\infty |\mu_t-\bl_t|^2\,dG(t) < \infty$.
\end{proof}

\phantomsection
\addcontentsline{toc}{section}{References}
\section*{Acknowledgements}
This research was supported by Vega Institute Foundation.

\phantomsection
\addcontentsline{toc}{section}{References}
\bibliographystyle{apalike}
\bibliography{prediction}

\end{document}